\documentclass[a4paper,reqno]{amsart}

\usepackage{amssymb}
\usepackage{latexsym}
\usepackage{amsmath}
\usepackage{hyperref}
\usepackage{color}

\usepackage[shortlabels]{enumitem}

   \def\dN{{\mathbb N}}

\def\cG{{\mathcal G}}   \def\cH{{\mathcal H}}

\def\R{\mathbb{R}}
\def\C{\mathbb{C}}

\def\ran{{\text{\rm ran\,}}}
\def\dom{{\text{\rm dom\,}}}
\def\min{{\rm min\,}}

 \usepackage[usestackEOL]{stackengine}
\usepackage{scalerel}
\def\moverline#1{\ThisStyle{%
  \setbox0=\hbox{$\SavedStyle#1$}%
  \stackengine{1.4\LMpt}{$\SavedStyle#1$}{\rule{\wd0}{.3\LMpt}}{0}{c}{F}{F}{S}%
}}
 \newcommand{\myoverline}[1]{\mkern 1mu\moverline{\mkern-1mu#1\mkern-1mu}\mkern 1mu}

\newtheorem{theorem}{Theorem}[section]
\newtheorem*{thm*}{Theorem}
\newtheorem{proposition}[theorem]{Proposition}
\newtheorem{corollary}[theorem]{Corollary}
\newtheorem{lemma}[theorem]{Lemma}

\theoremstyle{definition}
\newtheorem{definition}[theorem]{Definition}

\newtheorem{assumption}[theorem]{Assumption}

\numberwithin{equation}{section}

\title[]{Generalized boundary triples for adjoint pairs with applications to non-self-adjoint Schrödinger operators}

\author{Antonio Arnal}
\author{Jussi Behrndt}
\author{Markus Holzmann\textsuperscript{*}}
\thanks{\noindent\textsuperscript{*} Corresponding author,\quad E-mail:~{\tt holzmann@math.tugraz.at}}
\author{Petr Siegl}

\address{Institut f\"ur Angewandte Mathematik \\
 Technische Universit\"at Graz \\
 Steyrergasse 30\\
 A-8010 Graz\\
 Austria}
 \email{aarnalperez01@qub.ac.uk,behrndt@tugraz.at, holzmann@math.tugraz.at, siegl@tugraz.at}

\begin{document}

\dedicatory{
Happy Birthday, Henk$\,$!\\
We dedicate this note to our friend
and colleague Henk de Snoo on the occasion of his 80th birthday.}

\begin{abstract}
We extend the notion of generalized boundary triples and their Weyl functions from extension theory of symmetric operators  
to adjoint pairs of operators, and we provide criteria on the boundary parameters to induce closed operators with a nonempty resolvent set.
The abstract results are applied to Schr\"odinger operators with  complex $L^p$-potentials on bounded and unbounded 
Lipschitz domains with compact boundaries. 
\end{abstract}

\maketitle

\section{Introduction}

Abstract boundary value problems for symmetric and self-adjoint operators in Hilbert spaces can be efficiently treated within the framework of boundary triples and their Weyl functions, a technique that
is nowadays well developed and applied in various concrete and abstract settings, see, e.g., the monograph \cite{BHS20} for an introduction into this field,
some typical applications to differential operators, and further references.   Henk's contributions
on boundary triples and their generalizations have inspired many mathematicians in modern analysis,
differential equations, and spectral theory. In particular, abstract coupling techniques for boundary triples 
from \cite{DHMS00} or 
the notion of boundary relations and their Weyl families from \cite{DHMS06,DHMS09,DHMS12}, developed jointly with V.A.~Derkach, S.~Hassi, and M.M.~Malamud, have had a substantial impact on the field. 

The aim of this note is to consider a class of non-symmetric abstract boundary value problems in the context of 
adjoint pairs of operators. Such problems have their roots in the works
of M.I. Visik \cite{V52}, M.S. Birman \cite{B62}, and G. Grubb \cite{G68}, and have been 
further developed in the framework of (ordinary) boundary triples by L.I. Vainerman in \cite{V80} and 
V.E.~Lyantse and O.G.~Storozh in 
the monograph \cite{LS83}, see also 
\cite{BGHN17,BGW09,BHMNW09,BMNW08,HMM05,HMM13,MM97,MM99,MM02,M06} for more recent developments and, e.g., 
\cite{AB10,ABCE17,AEM17,AEKS14,A00,A12,BEHL18,EGC07,GM08,G08,LT77,LS22,LS23,M10,MPS16,P04,P08,P07,P13,W17,W13} for other closely related approaches and  typical  
applications.  In fact, our main objective is to introduce and to study the notion of generalized boundary triples and their Weyl functions for adjoint pairs of operators, extending the 
definition by V.A.~Derkach and M.M.~Malamud from \cite{DM95} (see also \cite{DHM20,DHM22}) from the symmetric to the non-symmetric setting. At the same time the present considerations can be viewed as a special case of the treatment in \cite{B25}, where the notion of quasi boundary triples and their Weyl functions for symmetric operators from \cite{BL07,BL12,BLLR18,BM14} was extended to the general framework of adjoint pairs under minimal assumptions on the boundary maps. Although our abstract treatment in Section~\ref{gbt-section}
in this sense is contained in \cite{B25}, many of the general results from \cite{B25} simplify substantially 
in their assumptions and their formulation. 

The abstract notion of generalized 
boundary triples for adjoint pairs turns out to be useful as it can be applied directly to Schr\"{o}dinger operators 
with complex potentials on Lipschitz domains $\Omega$. More precisely, in Section~\ref{diffopsec} we consider a class 
of non-selfadjoint relatively bounded perturbations $V\in L^p(\Omega)$ of the Laplacian and
provide a generalized boundary triple for the adjoint pair $\{-\Delta + V,-\Delta + \overline V\}$,
where the boundary maps are the Dirichlet and Neumann trace operators $\tau_D$ and $\tau_N$; here we rely on  
properties of the trace operators on Lipschitz domains with compact boundaries that can found in, e.g., \cite{BGM25,GM11} and are recalled in the appendix.
We apply this generalized boundary triple to obtain sufficient criteria for boundary parameters in $L^2(\partial\Omega)$ to induce closed realizations 
of $-\Delta + V$ and $-\Delta + \overline V$ with nonempty resolvent set in $L^2(\Omega)$. More precisely, if the complex potential 
$V\in L^p(\Omega)$ satisfies Assumption~\ref{assi},
we conclude in 
Corollary~\ref{corollary_Robin} that the Robin realization 
\begin{equation}\label{abab}
\begin{split}
 A_B&=-\Delta+ V,\\ 
 \dom A_B&=\bigl\{f\in H^{3/2}(\Omega):\tau_N f=B\tau_D f,\, (-\Delta +V)f\in L^2(\Omega)\bigr\},
\end{split}
 \end{equation}
is a closed operator in $L^2(\Omega)$ with a nonempty resolvent set; here it is assumed that 
the boundary parameter $B$ is a bounded everywhere defined operator in $L^2(\partial\Omega)$. 
Furthermore, if $A_0$ denotes the Neumann realization 
of $-\Delta+V$, and $\gamma,\widetilde\gamma$, and $M$ are the $\gamma$-fields and Weyl function, respectively, associated with the generalized boundary triple in Theorem~\ref{gbthurra}, then the Krein-type resolvent formula
\begin{equation}\label{Eq_Krein_formulaintrr}
(A_{B}-\lambda)^{-1}=(A_0-\lambda)^{-1}+\gamma(\lambda)\bigl(I-BM(\lambda)\bigr)^{-1}B\widetilde\gamma(\myoverline{\lambda})^*
\end{equation}
is valid for all $\lambda\in\rho(A_0)\cap\rho(A_{B})$. One of the key ingredients in the proofs is to 
ensure a decay of the norm of the Weyl function $M$ along the negative axis, so that the operator $I-BM(\lambda)$ in \eqref{Eq_Krein_formulaintrr}
admits a bounded everywhere defined inverse in $L^2(\partial\Omega)$. This technique is inspired by \cite{BLLR18}, where similar methods 
were developed for non-self-adjoint extensions of symmetric operators in the context of quasi boundary triples.
We point out that the functions in the domain of $A_B$ in \eqref{abab} exhibit 
$H^{3/2}$-regularity, which is natural for realizations of the Laplacian on (bounded) Lipschitz domains, such as Dirichlet and Neumann realizations;
cf. \cite{JK81,JK95} and \cite{BGM25} for more references and details. We also refer the reader to \cite{AGH19,BK08,BK12,BKL22,CK20,GM08,GM09-2,GM09,KRRS18,PP16} for related recent contributions on self-adjoint and non-self-adjoint 
Robin type boundary conditions. 

\subsection*{Notations}

For a linear operator $A$ its domain, kernel, and range are denoted by $\dom A$, $\ker A$, and $\ran A$, respectively. If $A$ is a closed operator in a Hilbert space, then the symbols $\rho(A)$, $\sigma(A)$, and $\sigma_p(A)$ are used for the resolvent set, the spectrum, and the point spectrum of $A$. Next, for an open set $\Omega \subset \mathbb{R}^n$ and $s \geq 0$ the sets $H^s(\Omega)$ are the $L^2$-based Sobolev spaces of order $s$, and for a bounded or unbounded Lipschitz domain $\Omega$ with compact boundary and $s \in [-1,1]$ the $L^2$-based Sobolev spaces on $\partial \Omega $ are $H^s(\partial \Omega)$; cf. \cite{M00} for their definitions. Finally, for a Banach space $X$ its dual space is $X^*$.

\section{Generalized boundary triples for adjoint pairs}\label{gbt-section}

Let $\cH$ be a separable Hilbert space and assume that 
$S$ and $\widetilde S$ are densely defined closed operators in $\cH$ which satisfy the identity
\begin{equation}\label{adjointis}
 (Sf,g)=(f,\widetilde S g),\qquad f\in\dom S,\, g\in\dom \widetilde S.
\end{equation}
It is clear that \eqref{adjointis} is equivalent to $\widetilde S\subset S^*$ or $S\subset \widetilde S^*$.
In the following a pair $\{S,\widetilde S\}$ with the property \eqref{adjointis} will be called an {\it adjoint pair}. Furthermore, we will make use of
another pair of operators $\{T,\widetilde T\}$ such that $ T\subset S^*$ and 
$\widetilde T\subset\widetilde S^*$ and, in addition, 
\begin{equation}\label{tja}
 \myoverline T=S^*\quad\text{and}\quad \myoverline{\widetilde T}=\widetilde S^*
\end{equation}
holds.
Note that \eqref{tja} is equivalent to $T^*=S$ and $\widetilde T^*=\widetilde S$.  A pair of operators $\{T,\widetilde T\}$ with the property \eqref{tja}
will be called a core of $\{S^*,\widetilde S^*\}$.

\subsection{Generalized boundary triples}
The next definition is a special case of \cite[Definition 2.1]{B25} and should be viewed as the natural generalization of the concept
of generalized boundary triples for symmetric operators from \cite{DM95}, see also \cite{BMN17,DHMS06,DHMS12}.

\begin{definition}\label{gbt}
Let $\{S,\widetilde S\}$ be an adjoint pair of operators and 
let $\{T,\widetilde T\}$ be a core of $\{S^*,\widetilde S^*\}$. A {\em generalized boundary triple} $\{\cG,(\Gamma_0,\Gamma_1),(\widetilde\Gamma_0,\widetilde\Gamma_1)\}$ for  $\{S,\widetilde S\}$ consists of a  Hilbert space $\cG$ and linear mappings
\begin{equation*}
 \Gamma_0,\Gamma_1:\dom T\rightarrow\cG\quad\text{and}\quad \widetilde\Gamma_0,\widetilde\Gamma_1:\dom \widetilde T\rightarrow\cG
\end{equation*}
such that the following holds:
\begin{enumerate}[\upshape (i)]
 \item the abstract {\em Green identity}
\begin{equation*}
 (Tf,g)_\cH-(f,\widetilde Tg)_\cH=(\Gamma_1 f,\widetilde \Gamma_0 g)_\cG-(\Gamma_0 f,\widetilde\Gamma_1 g)_\cG
\end{equation*}
is valid for all $f\in\dom T$ and $g\in\dom \widetilde T$,
\item the mappings $\Gamma_0:\dom T\rightarrow\cG$ and $\widetilde\Gamma_0:\dom \widetilde T\rightarrow \cG$ are surjective,
\item the operators $A_0:=T\upharpoonright\ker\Gamma_0$ and $\widetilde A_0:=\widetilde T\upharpoonright\ker\widetilde \Gamma_0$ satisfy
\begin{equation}\label{sternis}
A_0^*=\widetilde A_0\quad\text{and}\quad \widetilde A_0^*=A_0.
\end{equation}
\end{enumerate}
\end{definition}

Let $\{\cG,(\Gamma_0,\Gamma_1),(\widetilde\Gamma_0,\widetilde\Gamma_1)\}$ be a generalized boundary triple for $\{S,\widetilde S\}$ and let 
$\{T,\widetilde T\}$ be a core of $\{S^*,\widetilde S^*\}$. We briefly recall some immediate properties that follow from the more general treatment
in \cite{B25} and are known for the symmetric case from \cite{DM95}. 
We remark first that (iii) implies 
\begin{equation*}
\widetilde S\subset T\subset S^*\quad\text{and}\quad S\subset \widetilde T\subset \widetilde S^*.
\end{equation*}
According to \cite[Lemma 2.4]{B25} we have 
\begin{equation*}
  \dom S=\ker\widetilde\Gamma_0\cap\ker\widetilde\Gamma_1\quad\text{and}\quad
  \dom \widetilde S=\ker\Gamma_0\cap\ker\Gamma_1,
 \end{equation*}
 and $\ran(\Gamma_0,\Gamma_1)^\top$ and $\ran(\widetilde\Gamma_0,\widetilde\Gamma_1)^\top$ are both dense in $\cG\times \cG$ by \cite[Lemma 2.3]{B25}.
Furthermore, the mappings $\Gamma_0,\Gamma_1:\dom T\rightarrow\cG$ and $\widetilde\Gamma_0,\widetilde\Gamma_1:\dom\widetilde T\rightarrow\cG$
are closable with respect to the graph norm of $T$ and $\widetilde T$, respectively. 
It is also important to note that the pair $\{T,\widetilde T\}$ is not unique and may even coincide with $\{S^*,\widetilde S^*\}$, which is the case
if, e.g., 
\begin{equation}\label{deffinite}
\dim\bigl(\dom S^*/\dom\widetilde S\bigr)= \dim\bigl(\dom \widetilde S^*/\dom S\bigr)<\infty.
\end{equation}
In the special situation that $\{T,\widetilde T\}=\{S^*,\widetilde S^*\}$ it follows from \cite[Proposition~2.6]{B25} that 
$\ran(\Gamma_0,\Gamma_1)^\top=\cG\times\cG$ and $\ran(\widetilde\Gamma_0,\widetilde\Gamma_1)^\top=\cG\times\cG$, which implies that the restrictions
$A_0=T\upharpoonright\ker\Gamma_0$ and $\widetilde A_0=\widetilde T\upharpoonright\ker\widetilde \Gamma_0$ automatically satisfy \eqref{sternis}; cf. 
\cite{MM02}.
Thus, in this situation the notion of generalized boundary triples from Definition~\ref{gbt} reduces to the special case treated in 
\cite{BMNW08,MM02,LS83,V80}. In the following we will be interested in the case 
\begin{equation*}
\dim\bigl(\dom S^*/\dom\widetilde S\bigr)= \dim\bigl(\dom \widetilde S^*/\dom S\bigr)=\infty,
\end{equation*}
although our abstract discussion remains valid (in a simplified form) also in the finite dimensional case \eqref{deffinite}.

The next result is a variant of \cite[Theorem 2.7]{B25} and can be used to verify that a pair $\{T,\widetilde T\}$
is a core of the adjoints of suitable operators $S$ and $\widetilde S$, respectively. 

\begin{theorem}\label{ratemal-gbtversion}
 Let $\cH$ and $\cG$ be Hilbert spaces, let $T$ and $\widetilde T$ be operators in $\cH$, and assume that there are linear mappings 
 \begin{equation*}
  \Gamma_0,\Gamma_1:\dom T\rightarrow\cG\quad\text{and}\quad\widetilde\Gamma_0,\widetilde\Gamma_1:\dom \widetilde T\rightarrow\cG
 \end{equation*}
 such that the following holds:
\begin{enumerate}[\upshape (i)]
  \item the abstract Green identity 
  \begin{equation*}
 (Tf,g)_\cH-(f,\widetilde Tg)_\cH=(\Gamma_1 f,\widetilde \Gamma_0 g)_\cG-(\Gamma_0 f,\widetilde\Gamma_1 g)_\cG
\end{equation*}
is valid for all $f\in\dom T$ and $g\in\dom \widetilde T$,
  \item the mappings $\Gamma_0:\dom T\rightarrow\cG$ and 
  $\widetilde\Gamma_0:\dom \widetilde T\rightarrow\cG$ are surjective,
 \item the operators $A_0:=T\upharpoonright\ker\Gamma_0$ and $\widetilde A_0:=\widetilde T\upharpoonright\ker\widetilde \Gamma_0$ satisfy 
\begin{equation*}
A_0^*=\widetilde A_0\quad\text{and}\quad \widetilde A_0^*=A_0,
\end{equation*}
\item $\ker\Gamma_0\cap\ker\Gamma_1$ and $\ker\widetilde\Gamma_0\cap\ker\widetilde\Gamma_1$  are dense in $\cH$.
 \end{enumerate}
Then
\begin{equation*}
\begin{split}
 Sf:= \widetilde Tf,&\quad f\in \dom S=\ker\widetilde\Gamma_0\cap\ker\widetilde\Gamma_1,\\
 \widetilde Sg:=  Tg,&\quad g\in\dom \widetilde S=\ker\Gamma_0\cap\ker\Gamma_1,
\end{split}
\end{equation*}
are closed operators in $\cH$ and form an adjoint pair $\{S,\widetilde S\}$ such that $\{T,\widetilde T\}$ is a core of $\{S^*,\widetilde S^*\}$, and 
$\{\cG,(\Gamma_0,\Gamma_1),(\widetilde\Gamma_0,\widetilde\Gamma_1)\}$ is a generalized boundary triple for $\{S,\widetilde S\}$.
\end{theorem}

\subsection{$\gamma$-fields and Weyl functions}

In the following let $\{\cG,(\Gamma_0,\Gamma_1),(\widetilde\Gamma_0,\widetilde\Gamma_1)\}$ be a generalized boundary triple for $\{S,\widetilde S\}$ 
and assume that the resolvent sets of both operators $A_0=T\upharpoonright\ker\Gamma_0$ and 
$\widetilde A_0=\widetilde T\upharpoonright\ker\widetilde\Gamma_0$ are nonempty. In this case we have $\lambda\in\rho(A_0)$ 
if and only if $\myoverline\lambda\in\rho(\widetilde A_0)$.
Next, recall the direct sum decompositions 
\begin{equation*}
\begin{split}
 \dom T=\dom A_0\,\dot +\,\ker(T-\lambda)=\ker \Gamma_0\,\dot +\,\ker(T-\lambda),\qquad\lambda\in\rho(A_0),\\
 \dom \widetilde T=\dom \widetilde A_0\,\dot +\,\ker(\widetilde T-\mu)=\ker \widetilde\Gamma_0\,\dot +\,\ker(\widetilde T-\mu),\qquad\mu\in\rho(\widetilde A_0).
 \end{split}
 \end{equation*}
 
We introduce the $\gamma$-fields and Weyl functions following the ideas in \cite{B25,BL07,DHMS06,DM95,MM97,MM99,MM02}.
 
 \begin{definition}
 Let  $\{\cG,(\Gamma_0,\Gamma_1),(\widetilde\Gamma_0,\widetilde\Gamma_1)\}$ be a generalized boundary triple and assume 
 that $\rho(A_0)\not=\emptyset$ or, equivalently, $\rho(\widetilde A_0)\not=\emptyset$. 
\begin{enumerate}[\upshape (i)]
 \item The $\gamma$-fields $\gamma$ and $\widetilde\gamma$ are defined by 
 \begin{equation*}
 \begin{split}
  \gamma(\lambda)&:=\bigl(\Gamma_0\upharpoonright\ker(T-\lambda)\bigr)^{-1},\qquad \lambda\in\rho(A_0),\\
  \widetilde\gamma(\mu)&:=\bigl(\widetilde\Gamma_0\upharpoonright\ker(\widetilde T-\mu)\bigr)^{-1},\qquad \mu\in\rho(\widetilde A_0).
 \end{split}
 \end{equation*}
 \item The Weyl functions $M$ and $\widetilde M$ are defined by 
 \begin{equation*}
 \begin{split}
  M(\lambda)&:=\Gamma_1\bigl(\Gamma_0\upharpoonright\ker(T-\lambda)\bigr)^{-1}=\Gamma_1\gamma(\lambda),\qquad \lambda\in\rho(A_0),\\
  \widetilde M(\mu)&:=\widetilde\Gamma_1\bigl(\widetilde\Gamma_0\upharpoonright\ker(\widetilde T-\mu)\bigr)^{-1}=\widetilde\Gamma_1\widetilde\gamma(\mu),\qquad \mu\in\rho(\widetilde A_0).
 \end{split}
 \end{equation*}
\end{enumerate}
 \end{definition}

In the next proposition we collect some properties of the $\gamma$-fields and Weyl functions corresponding to a generalized boundary triple, see \cite[Proposition~3.3 and Proposition~3.4]{B25} for proofs.
 
\begin{proposition}\label{gam-m-prop}
 Let  $\{\cG,(\Gamma_0,\Gamma_1),(\widetilde\Gamma_0,\widetilde\Gamma_1)\}$ be a generalized boundary triple and assume 
 that $\rho(A_0)\not=\emptyset$ or, equivalently, $\rho(\widetilde A_0)\not=\emptyset$. 
Then the following holds for $\lambda\in\rho(A_0)$ and $\mu\in\rho(\widetilde A_0)$:
\begin{enumerate}[\upshape (i)]
 \item $\gamma(\lambda)$ and $\widetilde\gamma(\mu)$ are everywhere defined bounded operators from $\cG$ to $\cH$;
 \item the functions $\lambda\mapsto\gamma(\lambda)$ and $\mu\mapsto\widetilde\gamma(\mu)$ are holomorphic 
 on $\rho(A_0)$ and $\rho(\widetilde A_0)$, respectively, and one has
 \begin{equation*}
  \begin{split}
   \gamma(\lambda)&=\bigl(I+(\lambda-\nu)(A_0-\lambda)^{-1}\bigr)\gamma(\nu),\qquad \lambda,\nu\in\rho(A_0),\\
   \widetilde\gamma(\mu)&=\bigl(I+(\mu-\omega)(\widetilde A_0-\mu)^{-1}\bigr)\widetilde\gamma(\omega),\qquad \mu,\omega\in\rho(\widetilde A_0);
  \end{split}
\end{equation*}
 \item $\gamma(\lambda)^*$ and $\widetilde\gamma(\mu)^*$ are everywhere defined bounded operators from $\cH$ to $\cG$
 and  one has 
 \begin{equation*}
  \gamma(\lambda)^*=\widetilde\Gamma_1(\widetilde A_0-\myoverline\lambda)^{-1}\quad\text{and}\quad
  \widetilde\gamma(\mu)^*=\Gamma_1(A_0-\myoverline\mu)^{-1};
 \end{equation*}
 \item $M(\lambda)$ and $\widetilde M(\mu)$ are everywhere defined bounded operators in $\cG$, and for $f_\lambda\in\ker(T-\lambda)$ 
 and $g_\mu\in \ker(\widetilde T-\mu)$ one has
\begin{equation*}
 M(\lambda)\Gamma_0 f_\lambda =\Gamma_1 f_\lambda\quad\text{and}\quad
 \widetilde M(\mu)\widetilde\Gamma_0 g_\mu =\widetilde\Gamma_1 g_\mu;
\end{equation*}
\item $M(\lambda)=\widetilde M(\myoverline\lambda)^*$ and $\widetilde M(\mu)= M(\myoverline\mu)^*$ and one has
\begin{equation*}
\begin{split}
 M(\lambda)-\widetilde M(\mu)^*&=(\lambda-\myoverline\mu)\widetilde\gamma(\mu)^*\gamma(\lambda),\\
 M(\lambda)^*-\widetilde M(\mu)&=(\myoverline\lambda-\mu)\gamma(\lambda)^*\widetilde\gamma(\mu);
 \end{split}
\end{equation*}
\item
the functions $\lambda\mapsto M(\lambda)$ and $\mu\mapsto\widetilde M(\mu)$ are holomorphic on $\rho(A_0)$ and $\rho(\widetilde A_0)$, respectively, and for some fixed $\lambda_0,\mu_0\in\rho(A_0)\cap\rho(\widetilde A_0)$ one has
\begin{equation*}
\begin{split} 
 M(\lambda)&=\widetilde M(\lambda_0)^*+\widetilde\gamma(\lambda_0)^*(\lambda-\myoverline\lambda_0)\bigl(I+(\lambda-\lambda_0)(A_0-\lambda)^{-1}\bigr)\gamma(\lambda_0),\\
 \widetilde M(\mu)&=M(\mu_0)^*+\gamma(\mu_0)^*(\mu-\myoverline\mu_0)\bigl(I+(\mu-\mu_0)(\widetilde A_0-\mu)^{-1}\bigr)\widetilde\gamma(\mu_0).
\end{split} 
\end{equation*}
\end{enumerate}
\end{proposition}

\subsection{Closed extensions and their resolvents}
Next we introduce two families of operators in $\cH$ as restrictions of $T$ and $\widetilde T$ via abstract boundary conditions in $\cG$. 
 Let again
 $\{\cG,(\Gamma_0,\Gamma_1),(\widetilde\Gamma_0,\widetilde\Gamma_1)\}$ be a generalized boundary triple for the adjoint pair $\{S,\widetilde S\}$,
 where $\{T, \widetilde T\}$ is a core of $\{S^*,\widetilde S^*\}$.
 For linear operators $B$ and $\widetilde B$ in $\cG$ we define the restrictions $A_B$ of $T$ and $\widetilde A_{\widetilde B}$ of $\widetilde T$ via abstract boundary conditions in $\cG$ by 
 \begin{equation}\label{a123}
 \begin{split}
  A_B f&:= Tf,\qquad \dom A_B:=\bigl\{f\in\dom T: B\Gamma_1 f = \Gamma_0 f\bigr\},\\
  \widetilde A_{\widetilde B} g&:= \widetilde Tg,\qquad \dom \widetilde A_{\widetilde B}:=\bigl\{g\in\dom \widetilde T: \widetilde B\widetilde\Gamma_1 g 
  = \widetilde \Gamma_0 g\bigr\}.
 \end{split}
 \end{equation}
It follows from Green's identity that for a densely defined operator $B$ one has 
\begin{equation*}
 A_B\subset (\widetilde A_{B^*})^* \quad\text{and}\quad \widetilde A_{B^*}\subset (A_B)^*.
\end{equation*}
In the next lemma, we formulate an abstract version of the Birman-Schwinger principle to characterize eigenvalues of the extensions $A_B$ and $\widetilde A_{\widetilde B}$ via the Weyl functions; cf. \cite[Corollary~4.3]{B25} for a proof.

\begin{lemma}\label{bslem}
Let  $\{\cG,(\Gamma_0,\Gamma_1),(\widetilde\Gamma_0,\widetilde\Gamma_1)\}$ be a generalized boundary triple, assume 
 that $\rho(A_0)\not=\emptyset$ or, equivalently, $\rho(\widetilde A_0)\not=\emptyset$, and let  $M$ and $\widetilde M$ be the associated Weyl functions.
Then the following assertions hold for the operators $A_B$ and $\widetilde A_{\widetilde B}$ in \eqref{a123}, and all $\lambda\in\rho(A_0)$
and $\mu\in\rho(\widetilde A_0)$:
\begin{enumerate}[\upshape (i)]
\item $\lambda\in\sigma_p(A_{B})$ if and only if $\ker(I-BM(\lambda))\not=\{0\}$, and in this case 
 \begin{equation*}
  \ker(A_B-\lambda)=\gamma(\lambda)\ker(I-BM(\lambda)).
 \end{equation*}
\item $\mu\in\sigma_p(\widetilde A_{\widetilde B})$ if and only if 
 $\ker(I-\widetilde B\widetilde M(\mu))\not=\{0\}$, and in this case 
 \begin{equation*}
  \ker(\widetilde A_{\widetilde B}-\mu)=\widetilde\gamma(\mu)\ker(I-\widetilde B\widetilde M(\mu)).
 \end{equation*}
\end{enumerate}
\end{lemma}

In the next theorem we impose abstract conditions on $f,g\in\cH$, the $\gamma$-fields, Weyl functions, and parameters $B$ and $\widetilde B$  such that 
a Krein-type formula for the inverses of $A_{B}-\lambda$ and $\widetilde A_{\widetilde B}-\mu$ applied to $f$ and $g$, respectively, becomes meaningful; cf.
\cite[Theorem 4.4]{B25}.
These conditions will be made more explicit in Theorem~\ref{bssatz}.

\begin{theorem}\label{ojathm}
 Let  $\{\cG,(\Gamma_0,\Gamma_1),(\widetilde\Gamma_0,\widetilde\Gamma_1)\}$ be a generalized boundary triple and assume 
 that $\rho(A_0)\not=\emptyset$ or, equivalently, $\rho(\widetilde A_0)\not=\emptyset$. 
Let $\gamma,\widetilde\gamma$ and $M,\widetilde M$ be the associated $\gamma$-fields and Weyl functions, respectively.
Then the following assertions hold for the operators $A_B$ and $\widetilde A_{\widetilde B}$ in \eqref{a123}, and all $\lambda\in\rho(A_0)$
and $\mu\in\rho(\widetilde A_0)$:
\begin{enumerate}[\upshape (i)]
\item If $\lambda\not\in\sigma_p(A_B)$ and $f\in\cH$ is such that $\widetilde\gamma(\myoverline\lambda)^*f\in\dom B$ and $B\widetilde\gamma(\myoverline\lambda)^*f\in\ran(I-BM(\lambda))$,
then $f\in\ran(A_B-\lambda)$ and 
\begin{equation*}
 (A_B-\lambda)^{-1}f=(A_0-\lambda)^{-1}f+\gamma(\lambda)\bigl(I-BM(\lambda)\bigr)^{-1}B\widetilde\gamma(\myoverline\lambda)^*f.
\end{equation*}
\item  If $\mu\not\in\sigma_p(\widetilde A_{\widetilde B})$ and $g\in\cH$ is such that $\gamma(\myoverline\mu)^*g\in\dom \widetilde B$ and $\widetilde B\gamma(\myoverline\mu)^*g\in\ran(I-\widetilde B\widetilde M(\mu))$,
then $g\in\ran(\widetilde A_{\widetilde B}-\mu)$ and 
\begin{equation*}
 (\widetilde A_{\widetilde B}-\mu)^{-1}g=(\widetilde A_0-\mu)^{-1}g+\widetilde\gamma(\mu)\bigl(I-\widetilde B\widetilde M(\mu)\bigr)^{-1}\widetilde B\gamma(\myoverline\mu)^*g.
\end{equation*}
\end{enumerate}
\end{theorem}

As a corollary of Theorem~\ref{ojathm} we formulate the following result; cf. \cite[Theorem~4.6]{B25}.

\begin{corollary}
Let the assumptions be as in Theorem~\ref{ojathm}  and let
 $B$ be a densely defined operator in $\cG$ such that for some $\lambda\in\rho(A_0)$ one has 
 \begin{equation*}
  \widetilde\gamma(\myoverline\lambda)^*f\in\dom B\quad\text{and}\quad B\widetilde\gamma(\myoverline\lambda)^*f\in\ran(I-BM(\lambda))\quad\text{for all}\quad f\in\cH,
 \end{equation*}
and for some $\mu\in\rho(\widetilde A_0)$ one has 
 \begin{equation*}\gamma(\myoverline\mu)^*g\in\dom B^*\quad\text{and}\quad B^*\gamma(\myoverline\mu)^*g\in\ran(I-B^*\widetilde M(\mu))
   \quad\text{for all}\quad g\in\cH.
 \end{equation*}
Then  $A_{B}$  in \eqref{a123}
is a closed operator with a nonempty resolvent set and one has  $A_{B} = (\widetilde A_{B^*})^*$.
\end{corollary}

Now we recall a more direct and explicit criterion in terms of the Weyl function and the boundary operator $B$ such that $A_{B}$ becomes 
a closed operator with a nonempty resolvent set; cf. \cite[Corollary 4.9]{B25}.

\begin{theorem}\label{bssatz}
 Let  $\{\cG,(\Gamma_0,\Gamma_1),(\widetilde\Gamma_0,\widetilde\Gamma_1)\}$ be a generalized boundary triple and assume 
 that $\rho(A_0)\not=\emptyset$ or, equivalently, $\rho(\widetilde A_0)\not=\emptyset$. 
Let $\gamma,\widetilde\gamma$ and $M,\widetilde M$ be the associated $\gamma$-fields and Weyl functions, respectively. 

\begin{enumerate}[\upshape (i)]
\item
Assume that $B$ is a closable operator in $\cG$ such that $1\in\rho(B M(\lambda_0))$ for some $\lambda_0\in\rho(A_0)$ and $\ran(\Gamma_1\upharpoonright\ker\Gamma_0)\subset\dom B$.
Then  $A_{B}$  in \eqref{a123}
is a closed operator, $\lambda_0 \in \rho(A_B)$, and the Krein-type resolvent formula
\begin{equation*}
(A_{B}-\lambda)^{-1}=(A_0-\lambda)^{-1}+\gamma(\lambda)\bigl(I-BM(\lambda)\bigr)^{-1}B\widetilde\gamma(\myoverline{\lambda})^*
\end{equation*}
holds for all $\lambda\in\rho(A_0)\cap\rho(A_{B})$.

\item
Assume that $\widetilde B$ is a closable operator in $\cG$ such that $1\in\rho(\widetilde B \widetilde M(\mu_0))$ 
for some $\mu_0\in\rho(\widetilde A_0)$ and $\ran(\widetilde\Gamma_1\upharpoonright\ker\widetilde\Gamma_0)\subset\dom \widetilde B$.
Then  $\widetilde A_{\widetilde B}$  in \eqref{a123}
is a closed operator, $\mu_0 \in \rho(\widetilde A_{\widetilde B})$, and  the Krein-type resolvent formula
\begin{equation*}
 (\widetilde A_{\widetilde B}-\mu)^{-1}=(\widetilde A_0-\mu)^{-1}+\widetilde\gamma(\mu)\bigl(I-\widetilde B\widetilde M(\mu)\bigr)^{-1}\widetilde B\gamma(\myoverline\mu)^*
\end{equation*}
holds for all $\mu\in\rho(\widetilde A_0)\cap\rho(\widetilde A_{\widetilde B})$.
\end{enumerate}
\end{theorem}

\begin{proof}[Sketch of the proof of \textup{(i)}]
First, the assumption $\ran(\Gamma_1\upharpoonright\ker\Gamma_0)\subset\dom B$ and Proposition~\ref{gam-m-prop}~(iii) 
imply that $B\widetilde\gamma(\myoverline{\lambda}_0)^*$ is an everywhere defined operator from $\cH$ into $\cG$. Moreover, as $B$ is closable and $\widetilde\gamma(\myoverline{\lambda}_0)^*$ is bounded, $B\widetilde\gamma(\myoverline{\lambda}_0)^*$ is closed and hence bounded by the closed graph theorem. Next, using the assumption $1\in\rho(BM(\lambda_0))$ we conclude 
$\lambda_0\not\in\sigma_p(A_B)$ from Lemma~\ref{bslem}~(i) and, moreover, 
we also have $\ran(I-BM(\lambda_0))=\cG$.  Now it follows from Theorem~\ref{ojathm}~(i)
that $\ran(A_B-\lambda_0)=\cH$ and 
\begin{equation*}
 (A_B-\lambda_0)^{-1}=(A_0-\lambda_0)^{-1}+\gamma(\lambda_0)\bigl(I-BM(\lambda_0)\bigr)^{-1}B\widetilde\gamma(\myoverline\lambda_0)^*.
\end{equation*}
Our assumptions ensure that the right hand side is a bounded operator in $\cH$ and hence $(A_B-\lambda_0)^{-1}$ is also bounded, which implies that $A_B$ is closed 
and $\lambda_0\in\rho(A_B)$. Finally, the same arguments as in  \cite[Proof of Theorem 4.7]{B25} show that the Krein-type resolvent formula
\begin{equation*}
(A_{B}-\lambda)^{-1}=(A_0-\lambda)^{-1}+\gamma(\lambda)\bigl(I-BM(\lambda)\bigr)^{-1}B\widetilde\gamma(\myoverline{\lambda})^*
\end{equation*}
holds for all $\lambda\in\rho(A_0)\cap\rho(A_{B})$.
\end{proof}

\section{Schr\"odinger operators with unbounded complex potentials on Lipschitz domains with compact boundaries}\label{diffopsec}

In this section we apply
the abstract notion of generalized 
boundary triples and their Weyl functions for adjoint pairs to Schr\"{o}dinger operators 
with complex potentials $V$ on Lipschitz domains. Here we treat the case of bounded Lipschitz domains $\Omega$ and exterior Lipschitz domains $\Omega$ (that is, complements of the closures of bounded Lipschitz domains) at the same time; cf. Assumption~\ref{assi}~(i) below and Appendix~\ref{appoja}. Using the Dirichlet and Neumann trace operators $\tau_D$ and $\tau_N$ 
on the space $H^{3/2}_\Delta(\Omega)$ (see \eqref{hs} below) we
provide a generalized boundary triple for $\{-\Delta + V,-\Delta + \overline V\}$ in Theorem~\ref{gbthurra} and 
conclude sufficient criteria for Robin-type boundary conditions $B \tau_D f= \tau_N f$ in Corollary~\ref{corollary_Robin}
to induce closed realizations with nonempty resolvent set in $L^2(\Omega)$. 

The following assumption is crucial in this section.

\begin{assumption}\label{assi}

Let $n\in\dN$ with $n\geq 2$.
\begin{enumerate}[\upshape (i)]
	\item \label{asssi.Om} 
	Let $\Omega \subset \R^n$ be a bounded Lipschitz domain or an unbounded Lipschitz domain with compact boundary $\partial\Omega$ (i.e. $\Omega = \mathbb{R}^n \setminus \overline{\Omega_0}$ with $\Omega_0$ being a bounded Lipschitz domain) and let the outward unit normal be denoted by $\nu$.
	\item \label{asssi.V} 
	Let $V: \Omega \to \C$ be a measurable function such that $V \in L^p(\Omega)$ with
	$p \ge 2n/3$ if $n>3$ and $p>2$ if $n=2,3$.
\end{enumerate}
\end{assumption}

 The $L^2$-based Sobolev spaces on $\Omega$ will be denoted by $H^s(\Omega)$, $s\geq 0$, 
and we shall also make use of the Hilbert spaces
\begin{equation}\label{hs}
H^s_\Delta(\Omega):=\bigl\{f\in H^s(\Omega):\Delta f\in L^2(\Omega)\bigr\},\quad s\geq 0,
\end{equation}
equipped with the norms induced by
\begin{equation}\label{gn}
 (f,g)_{H^s_\Delta(\Omega)}:= (f,g)_{H^s(\Omega)} +(\Delta f,\Delta g)_{L^2(\Omega)},\quad f,g\in H^s_\Delta(\Omega).
\end{equation}
It is clear that for $s\geq 2$ the spaces $H^s_\Delta(\Omega)$ coincide with $H^s(\Omega)$.

In the next lemma we show that the Sobolev spaces $H^s(\Omega)$ are contained in the domain of the maximal multiplication operator in $L^2(\Omega)$ induced by some function $W \in L^p(\Omega)$ under suitable assumptions on $s$ and $p$. In particular, if $V=W$ satisfies Assumption~\ref{assi}, then $H^{3/2}(\Omega)$ is contained in the domain of the maximal multiplication operator induced by $V$ in $L^2(\Omega)$.
The proof is similar as the proof of \cite[Proposition 3.8]{BS19}.

\begin{lemma}\label{boundlem}
Let Assumption~\ref{assi} \ref{asssi.Om} be satisfied, let $s > 0$, let $W \in L^p(\Omega)$ such that $p \geq n/s$ if $n > 2s$ and $p>2$ if $n \leq 2s$,
and let $f \in H^s(\Omega)$. Then $W f,\myoverline W f \in L^2(\Omega)$ 
%\begin{equation*}
% Vf \in L^2(\Omega)\quad\text{for all }f\in H^{3/2}(\Omega)
%\end{equation*}
	and for every $\varepsilon > 0$ there exists $C_\varepsilon > 0$ such that
	\begin{equation}\label{naschau}
		\| Wf \|_{L^2(\Omega)}=\| \myoverline Wf \|_{L^2(\Omega)} \le \varepsilon \| f \|_{H^{s}(\Omega)} + C_\varepsilon \| f \|_{L^{2}(\Omega)}, \quad  f \in H^{s}(\Omega).
	\end{equation}
In particular, if $V\in L^p(\Omega)$ is as in Assumption~\ref{assi} \ref{asssi.V}, then the above statements are true for $s = 3/2$ and $W=V$, moreover, 
\begin{equation}\label{jaja1}
\begin{split}
  H^{3/2}_\Delta(\Omega)&=\bigl\{f\in H^{3/2}(\Omega):-\Delta f+Vf \in L^2(\Omega)\bigr\}\\
  &=\bigl\{f\in H^{3/2}(\Omega):-\Delta f+\myoverline V f \in L^2(\Omega)\bigr\}.
  \end{split}
 \end{equation}
\end{lemma}

\begin{proof}
Let $W\in L^p(\Omega)$ and let 
\begin{equation*}
 W_m(x)=\begin{cases} W(x), & \text{if} \;\; \vert W(x)\vert\leq m, \\ 0, & \text{if} \;\; \vert W(x)\vert > m, \end{cases}\qquad m\in\dN. 
\end{equation*}
Then we have $W_m\in L^p(\Omega)$, $m\in\dN$, and $\Vert W-W_m\Vert_{L^p(\Omega)}\rightarrow 0$ as $m\rightarrow\infty$.
For $f\in H^{s}(\Omega)$ we will make use of the estimate  
\begin{equation*}
 \| f \|_{L^q(\Omega)} \le C_{q, s} \| f \|_{H^{s}(\Omega)},\qquad q  \in 
 		\begin{cases}
 			[2, 2n/(n - 2s)], &\text{if} \;\; n>2s,\\
 			[2, \infty), &\text{if} \;\; n\leq 2s;
 		\end{cases}
\end{equation*}
cf. \cite[Theorem 8.12.6.I]{B12}. If $n>2s$ we use the generalized H\"older inequality with $1/p + 1/q=1/2$ and $2\leq q\leq 2n/(n-2s)$ (and hence $p\geq n/s$)
and obtain
\begin{equation*}
\begin{split}
 \Vert (W-W_m) f\Vert_{L^2(\Omega)}&\leq \Vert W-W_m \Vert_{L^{p}(\Omega)} \Vert f\Vert_{L^{q}(\Omega)} \\
 &\leq C_{q, s} \Vert W-W_m \Vert_{L^{p}(\Omega)} \Vert f\Vert_{H^{s}(\Omega)};
\end{split}
\end{equation*}
the same estimate holds also for $n \leq 2 s$ with $2\leq q<\infty$ (and hence $p>2$). Therefore, in both cases we conclude 
\begin{equation*}
\begin{split}
 \Vert W f\Vert_{L^2(\Omega)}&\leq \Vert (W-W_m) f\Vert_{L^2(\Omega)}+\Vert W_m f\Vert_{L^2(\Omega)}\\
 &\leq C_{q, s} \Vert W-W_m \Vert_{L^{p}(\Omega)} \Vert f\Vert_{H^{s}(\Omega)}+ m\Vert  f\Vert_{L^2(\Omega)},
\end{split}
\end{equation*}
which implies \eqref{naschau}.

Finally, if $V$ satisfies Assumption~\ref{assi}~(i), then the assumptions of this lemma are fulfilled for $s = 3/2$. Since $Vf,\myoverline Vf \in L^2(\Omega)$ for any $f\in H^{3/2}(\Omega)$ it is clear from \eqref{hs} with $s=3/2$ 
that \eqref{jaja1} holds.
\end{proof}
We recall that the Dirichlet Laplacian is defined by
\begin{equation}\label{diri}
H_D = -\Delta ,\quad \dom H_D=\bigl\{f\in H^{3/2}_\Delta(\Omega):\tau_D f=0\bigr\},
\end{equation}
and the Neumann Laplacian is defined by
\begin{equation}\label{dasindsie2}
H_N = -\Delta ,\quad \dom H_N=\bigl\{f\in H^{3/2}_\Delta(\Omega):\tau_N f=0\bigr\},
\end{equation}
where the Dirichlet trace 
\begin{equation}\label{diri3}
\tau_D:H^{3/2}_\Delta(\Omega)\rightarrow H^1(\partial\Omega)\subset L^2(\partial\Omega)
\end{equation}
and the  Neumann trace 
\begin{equation}\label{neumi3}
\tau_N:H^{3/2}_\Delta(\Omega)\rightarrow L^2(\partial\Omega)
\end{equation}
are as in \eqref{eq:tauD.rest.HsDelta} and \eqref{eq:tauN.bgm.ext}, respectively, with $s=3/2$. Both operators $H_D$ and $H_N$ in \eqref{diri}--\eqref{dasindsie2}
are self-adjoint and nonnegative in $L^2(\Omega)$; they coincide with the self-adjoint operators associated to 
the densely defined closed nonnegative forms 
\begin{equation}\label{hD.def}
\mathfrak{h}_D[f] = \|\nabla f\|_{L^2(\Omega)}^2, \quad \dom \mathfrak{h}_D = H^1_0(\Omega),
\end{equation}
and
\begin{equation}\label{hN.def}
\mathfrak{h}_N[f] = \|\nabla f\|_{L^2(\Omega)}^2, \quad \dom \mathfrak{h}_N = H^1(\Omega),
\end{equation}
via the first representation theorem \cite[Theorem~VI.2.1]{K66}, see, e.g., \cite[Theorem~6.9 and Theorem~6.10]{BGM25} for more details and \cite{JK81,JK95}
for the $H^{3/2}$-regularity of the operator domains.
In the following let $V\in L^p(\Omega)$ be as in Assumption~\ref{assi} and consider the differential expressions
\begin{equation*}
	-\Delta +V \quad\text{and}\quad -\Delta +\overline V. 
\end{equation*}
We define the corresponding  minimal operator realizations in $L^2(\Omega)$ by
\begin{equation}\label{sss}
\begin{split}
 S=-\Delta +\myoverline V,&\qquad \dom S= H^2_0(\Omega),\\
 \widetilde S=-\Delta +V,&\qquad \dom \widetilde S= H^2_0(\Omega),
 \end{split}
\end{equation}
and we shall also make use of the operators
\begin{equation}\label{ttt}
	\begin{split}
		T=-\Delta +V,&\qquad \dom T = H^{3/2}_{\Delta}(\Omega),\\
		\widetilde T=-\Delta + \myoverline V,&\qquad \dom \widetilde T = H^{3/2}_{\Delta}(\Omega).
	\end{split}
\end{equation}
It is not difficult to check that $\{S,\widetilde S\}$ form an adjoint pair and that $\{T,\widetilde T\}$ are well-defined; cf. \eqref{jaja1}.
In the proof of Theorem~\ref{gbthurra} below it will turn out en passant that $\{T,\widetilde T\}$ is a core of $\{S^*,\widetilde S^*\}$.

In the following we will make use of Theorem~\ref{ratemal-gbtversion} to construct a generalized boundary triple for the adjoint pair $\{S,\widetilde S\}$ in \eqref{sss}. 
To this end, we again consider the Dirichlet and Neumann trace operators $\tau_D$ and $\tau_N$ from \eqref{diri3}--\eqref{neumi3}, 
and choose the linear mappings
\begin{equation}\label{taus}
	\Gamma_0 = \widetilde\Gamma_0 = \tau_N\quad\text{and} \quad \Gamma_1 = \widetilde\Gamma_1 = \tau_D
\end{equation}
with domain $\dom T = \dom \widetilde T = H_{\Delta}^{3/2}(\Omega)$.
\begin{theorem}\label{gbthurra}
Let Assumption~\ref{assi} be satisfied. Consider the linear mappings
	\begin{equation*}
		\Gamma_0, \Gamma_1 : \dom T \to L^2(\partial\Omega)\quad\text{and} \quad\widetilde\Gamma_0, \widetilde\Gamma_1 : \dom \widetilde T \to L^2(\partial\Omega)
	\end{equation*}
	given by \eqref{taus}. Then $\{L^2(\partial\Omega), (\Gamma_0, \Gamma_1), (\widetilde\Gamma_0, \widetilde\Gamma_1)\}$ is a generalized boundary triple 
	for the adjoint pair $\{S,\widetilde S\}$ such that 
	\begin{equation}\label{dasindsie}
	\begin{split}
	 &A_0=T\upharpoonright\ker\Gamma_0 = -\Delta + V,\quad \dom A_0=\bigl\{f\in H^{3/2}_\Delta(\Omega):\tau_N f=0\bigr\},\\
	 &\widetilde A_0=\widetilde T\upharpoonright\ker\widetilde \Gamma_0 = -\Delta + \myoverline V,\quad \dom \widetilde A_0=\bigl\{g\in H^{3/2}_\Delta(\Omega):\tau_N g=0\bigr\},
	 \end{split}
	\end{equation}
	coincide with the Neumann realizations of $-\Delta + V$ and $-\Delta + \myoverline V$, respectively.
	Moreover, $A_0$ and $\widetilde A_0$ are closed operators in $L^2(\Omega)$ and there exists $\xi_1<0$ such that $(-\infty,\xi_1)\subset\rho(A_0)\cap\rho(\widetilde A_0)$.
\end{theorem}

\begin{proof}
We will verify that the operators $T$ and $\widetilde T$ in \eqref{ttt} and the boundary mappings in \eqref{taus} satisfy the conditions (i)-(iv) in 
Theorem~\ref{ratemal-gbtversion}. In fact, for $f,g\in\dom T= \dom \widetilde T = H^{3/2}_{\Delta}(\Omega)$ we have $Vf,\myoverline Vg\in L^2(\Omega)$ and hence
\begin{equation}\label{ggg}
\begin{split}
 (T f, g)_{L^2(\Omega)} - (f, \widetilde T g)_{L^2(\Omega)} &=(-\Delta f + Vf, g)_{L^2(\Omega)} - (f, -\Delta g +\myoverline Vg)_{L^2(\Omega)}\\
 &=(-\Delta f , g)_{L^2(\Omega)} - (f, -\Delta g )_{L^2(\Omega)}\\
 &=( \tau_D f, \tau_N g)_{L^2(\partial\Omega)} - ( \tau_N f, \tau_D g)_{L^2(\partial\Omega)}\\
 &=(\Gamma_1 f,\widetilde\Gamma_0 g)_{L^2(\partial\Omega)} - ( \Gamma_0 f, \widetilde\Gamma_1 g)_{L^2(\partial\Omega)},
\end{split}
\end{equation}
where we have used \eqref{greeni}. Thus, (i) in Theorem~\ref{ratemal-gbtversion} holds. It is clear from Theorem~\ref{tracen} that 
$\ran\Gamma_0=\ran\widetilde\Gamma_0=L^2(\partial\Omega)$ and hence also (ii) is satisfied. In order to check (iii) 
we use the fact that the Neumann Laplacian $H_N$ in \eqref{dasindsie2}
is self-adjoint and nonnegative in $L^2(\Omega)$. From Lemma~\ref{boundlem} and \eqref{boundn} we conclude that  
for every $\varepsilon > 0$ there exists $C_\varepsilon > 0$ such that
	\begin{equation*}
		\| Vf \|_{L^2(\Omega)} \le \varepsilon \| H_N f \|_{L^{2}(\Omega)} + C_\varepsilon \| f \|_{L^{2}(\Omega)}, \quad  f \in \dom H_N.
	\end{equation*}
In other words, $V$ and $\myoverline V$ are both relatively bounded with respect to $H_N$ with bound smaller than $1$ (in fact, with bound 0) and hence it follows from 
 \cite[$\S\,$3, Theorem~8.9]{EE87} that the	
restrictions 
$A_0=T\upharpoonright\ker\Gamma_0$ and $\widetilde A_0=\widetilde T\upharpoonright\ker\widetilde \Gamma_0$ in \eqref{dasindsie} 
are closed operators with nonempty resolvent sets and there exists $\xi_1<0$ such that $(-\infty,\xi_1)\subset\rho(A_0)\cap\rho(\widetilde A_0)$.
Green's identity \eqref{ggg} yields that $A_0\subset \widetilde A_0^*$ and $\widetilde A_0\subset A_0^*$ and hence also 
\begin{equation*}
 A_0-\lambda\subset \widetilde A_0^*-\lambda \quad\text{and}\quad \widetilde A_0-\lambda \subset A_0^*-\lambda,\quad \lambda\in (-\infty,\xi_1).
\end{equation*}
As $A_0-\lambda$ and $\widetilde A_0-\lambda$ are bijective we conclude  $A_0 = \widetilde A_0^*$ and $\widetilde A_0 = A_0^*$, that is, (iii) in 
Theorem~\ref{ratemal-gbtversion} holds. Finally, condition (iv) is satisfied as $C_0^{\infty}(\Omega) \subset \ker\Gamma_0 \cap \ker\Gamma_1$ and 
 $C_0^{\infty}(\Omega) \subset \ker\widetilde\Gamma_0 \cap \ker\widetilde\Gamma_1$. 

Now it follows from Theorem~\ref{ratemal-gbtversion} that the operators
\begin{equation}\label{trest}
	T\upharpoonright\ker\Gamma_0\cap\ker\Gamma_1 \quad\text{and}\quad 
	 \widetilde T\upharpoonright\ker\widetilde \Gamma_0\cap\ker\widetilde\Gamma_1
\end{equation}
form an adjoint pair and $\{L^2(\partial\Omega), (\Gamma_0, \Gamma_1), (\widetilde\Gamma_0, \widetilde\Gamma_1)\}$ is a generalized boundary triple.
One finds with Lemma~\ref{reglem} that the domains of the operators in \eqref{trest} are $H^2_0(\Omega)$ and hence these restrictions coincide with
the minimal operators in \eqref{sss}. Therefore, $\{T,\widetilde T\}$ is a core of $\{S^*,\widetilde S^*\}$ and 
$\{L^2(\partial\Omega), (\Gamma_0, \Gamma_1), (\widetilde\Gamma_0, \widetilde\Gamma_1)\}$ is a generalized boundary triple for $\{S,\widetilde S\}$.
\end{proof}

In the next proposition we identify the operators $A_0$ and $\widetilde A_0$ in \eqref{dasindsie} as representing operators of the closed
sectorial forms \eqref{formi1}--\eqref{formi2} below, and thus $A_0$ and $\widetilde A_0$ are automatically both $m$-sectorial. In addition, 
we provide, as a variant of \cite[Lemma~VI.3.1]{K66}, a useful representation of their 
resolvents in terms of (the square root) of the resolvent of the Neumann Laplacian $H_N$ in \eqref{dasindsie2}.
\begin{proposition}\label{prop:wt.A0.res}
Let Assumption~\ref{assi} be satisfied.  Let the operators $A_0,\widetilde A_0$, and $H_N$ be as in \eqref{dasindsie} and \eqref{dasindsie2}, respectively. Then $A_0$ and 
$\widetilde A_0$ are both $m$-sectorial operators associated with the closed sectorial forms
\begin{equation}\label{formi1}
	\mathfrak{a}_0[f] := \|\nabla f\|_{L^2(\Omega)}^2  + \int_\Omega V|f|^2 \, \mathrm{d} x, \quad \dom \mathfrak{a}_0 = H^1(\Omega),
\end{equation} 
and
\begin{equation}\label{formi2}
	\widetilde{\mathfrak{a}}_0[f] := \|\nabla f\|_{L^2(\Omega)}^2  + \int_\Omega \overline V|f|^2 \, \mathrm{d} x, \quad \dom \widetilde{\mathfrak{a}}_0 = H^1(\Omega),
\end{equation} 
respectively.
Moreover, there exists $\xi_2<0$ such that for all $\lambda\in (-\infty,\xi_2)$ one has $\lambda\in\rho(A_0)\cap\rho(\widetilde A_0)$ and there is a bounded operator $C_1(\lambda)$ with 
$\|C_1(\lambda)\| \leq 1/2$ such that
\begin{equation}\label{tA0.res}
	(A_0-\lambda)^{-1}= (H_N-\lambda)^{-1/2} (I+C_1(\lambda))^{-1}(H_N-\lambda)^{-1/2}
\end{equation}
and
\begin{equation}\label{tA0.res2}
	(\widetilde A_0-\lambda)^{-1}= (H_N-\lambda)^{-1/2} (I+C_1(\lambda)^*)^{-1}(H_N-\lambda)^{-1/2}.
\end{equation}
\end{proposition}

\begin{proof}
Recall the definition of the form $\mathfrak{h}_N$ in \eqref{hN.def} and that  the Neumann Laplacian in \eqref{dasindsie2}
is the self-adjoint operator associated to $\mathfrak{h}_N$ via the first representation theorem.
Next, we consider the form
\begin{equation*}
\mathfrak{v}[f] := \int_\Omega V|f|^2 \, \mathrm{d} x, \quad \dom \mathfrak{v} := \dom  \mathfrak{h}_N,
\end{equation*}
and employ Lemma~\ref{boundlem} for $W = |V|^{1/2} \in L^{2p}(\Omega)$ and $s=1$. It follows that for any $\delta>0$, there exists $\widetilde C_\delta>0$ such that
\begin{equation}\label{V.rel.bdd}
|\mathfrak{v}[f]| \leq \||V|^\frac12 f\|_{L^2(\Omega)}^2 \leq \delta \mathfrak{h}_N[f] + \widetilde C_\delta \|f\|_{L^2(\Omega)}^2, \quad f \in \dom  \mathfrak{h}_N.
\end{equation}
This shows that the form $\mathfrak{v}$ is a relatively bounded perturbation of $\mathfrak{h}_N$ with bound~$0$. By \cite[Theorem~VI.3.4]{K66}, the form 
$\mathfrak{a}_0= \mathfrak{h}_N + \mathfrak{v}$ is closed and sectorial, and hence it defines an $m$-sectorial operator $\widehat A_0$. Via the first Green's identity \eqref{usethissomewhere}, one verifies that $A_0 \subset \widehat A_0$. Since there exists $\lambda_0<0$ such that 
$\lambda_0 \in \rho(A_0) \cap \rho(\widehat A_0)$, it follows that $A_0 =\widehat A_0$. This justifies the first claim.
	
To show \eqref{tA0.res}, we slightly adjust the proof of \cite[Theorem~VI.3.2]{K66}. To this end, recall first that by the second representation theorem, see \cite[Theorem~VI.2.23 and Problem~VI.2.25]{K66}, we have for all $\lambda <0$ that 
$$\dom \mathfrak h_N  = \dom H_N^{1/2} = \dom (H_N-\lambda)^{1/2}$$ 
and 
\begin{equation}\label{hn.2nd}
(\mathfrak{h}_N-\lambda)[f,g] = ((H_N-\lambda)^{1/2} f, (H_N-\lambda)^{1/2} g)_{L^2(\Omega)}, \quad f,g \in \dom \mathfrak h_N.
\end{equation}
Next, it follows from  \eqref{V.rel.bdd} with $\delta=1/4$ that for all $\lambda < \lambda_1:= -4 \widetilde C_{1/4}$,
\begin{equation*}
|\mathfrak{v}[f]|\leq \frac 14 (\mathfrak{h}_N-\lambda)[f] +\left(\frac{1}{4}\lambda + \widetilde C_{1/4}\right)
\Vert f\Vert^2_{L^2(\Omega)}\leq \frac 14 (\mathfrak{h}_N-\lambda)[f] , \quad f \in \dom \mathfrak h_N.
\end{equation*}
Hence \cite[Lemma~VI.3.1]{K66} yields that for all $\lambda<\lambda_1$ there exists a bounded operator $C_1(\lambda)$ with $\|C_1(\lambda)\| \leq 1/2$ such that
\begin{equation}\label{v.2nd}
\mathfrak{v}[f,g] = \bigl(C_1(\lambda) (H_N-\lambda)^{1/2} f, (H_N-\lambda)^{1/2} g\bigr)_{L^2(\Omega)}, \quad f,g \in \dom \mathfrak h_N.
\end{equation}
Combing \eqref{hn.2nd} and \eqref{v.2nd}, we obtain for all $\lambda < \lambda_1$ and all  $f,g \in \dom \mathfrak h_N$ that
\begin{equation*}
\begin{aligned}
(\mathfrak{a}_0-\lambda)[f,g] &= (\mathfrak{h}_N-\lambda)[f,g] + \mathfrak{v}[f,g] 
\\
& =  \bigl((I+C_1(\lambda)) (H_N-\lambda)^{1/2}  f, (H_N-\lambda)^{1/2} g \bigr)_{L^2(\Omega)}.
\end{aligned}
\end{equation*}
Let further $f \in \dom  A_0  \subset \dom \mathfrak h_N$. Then for all $g \in \dom \mathfrak h_N$, 
\begin{equation*}
(( A_0-\lambda)f, g)_{L^2(\Omega)} = \bigl( (I+C_1(\lambda)) (H_N-\lambda)^{1/2} f, (H_N-\lambda)^{1/2} g \bigr)_{L^2(\Omega)},
\end{equation*}
and hence, since $(H_N-\lambda)^{1/2}$ is self-adjoint, 
\begin{equation*}
(A_0-\lambda)f =	(H_N-\lambda)^{1/2} (I+C_1(\lambda))(H_N-\lambda)^{1/2} f, \quad f \in \dom  A_0.
\end{equation*}
Therefore, we have shown that
\begin{equation}\label{wtA0.incl}
 A_0-\lambda \subset  (H_N-\lambda)^{1/2} (I+C_1(\lambda))(H_N-\lambda)^{1/2}.
\end{equation}
Recall that $\|C_1(\lambda)\|\leq 1/2$ for all $\lambda<\lambda_1$, thus $I+C_1(\lambda)$, and therefore also the operator on the right hand side of \eqref{wtA0.incl}, is boundedly invertible for such $\lambda$. On the other hand, since $A_0-\lambda$ is boundedly invertible for all $\lambda < \lambda_0$ (see above), we arrive at
\begin{equation}\label{wtA0.eq}
A_0-\lambda =  (H_N-\lambda)^{1/2} (I+C_1(\lambda))(H_N-\lambda)^{1/2}
\end{equation}
for all $\lambda < \xi_2:=\min\{\lambda_0,\lambda_1\}$. Finally, \eqref{tA0.res} follows by inverting \eqref{wtA0.eq}, and the assertions for 
$\widetilde{\mathfrak{a}}_0$ and $\widetilde A_0$ follow in the same way by replacing $V$ with $\overline V$ and taking adjoints.
\end{proof}

In the rest of this section we use the triple 
$\{L^2(\partial\Omega), (\Gamma_0, \Gamma_1), (\widetilde\Gamma_0, \widetilde\Gamma_1)\}$ to show that restrictions of $T$ that satisfy Robin-type boundary conditions $\tau_N f=B\tau_D f$ for a bounded operator $B$ in $L^2(\partial \Omega)$ are closed and have nonempty resolvent set, and we provide an explicit Krein-type resolvent formula for these operators; cf. Corollary~\ref{corollary_Robin}. For that purpose we use decay properties of the associated Weyl function; for this we need the following preparatory lemma. 

\begin{lemma} \label{lemma_bounded}
Let Assumption~\ref{assi} be satisfied.  Then there exists $\xi_3\leq-1$ such that for all $\lambda \in (-\infty,\xi_3)$ and all $\delta>0$ the densely defined operator $$\Gamma_1 (H_N - \lambda)^{-1/4 - \delta} = \widetilde{\Gamma}_1 (H_N - \lambda)^{-1/4 - \delta}$$ admits an everywhere defined bounded extension
\begin{equation*}
C_2(\lambda): L^2(\Omega) \rightarrow L^2(\partial \Omega),
\end{equation*}
 which is uniformly bounded in $\lambda \in (-\infty,\xi_3)$.
\end{lemma}
\begin{proof}
  Throughout the proof, we assume that $\lambda \leq  \xi_1 -1$, where $\xi_1<0$ is as in Theorem~\ref{gbthurra}.
  First, we claim for any $a \geq 0$ that
  \begin{equation} \label{resolvent_power_bounded}
    (H_N - \lambda)^{-a}: L^2(\Omega) \rightarrow H^{\min{\{1,2a\}}}(\Omega)
  \end{equation}
  is well-defined and uniformly bounded in $\lambda \leq  \xi_1 -1$. For $a=0$ this is clearly true. Next, we consider the case $a=1/2$. 
  Let $f \in L^2(\Omega)$ and let $\mathfrak h_N$ be the form in \eqref{hN.def}. Then one gets with the help of the second representation theorem \cite[Theorem~VI.2.23]{K66} applied for the nonnegative operator $H_N - \lambda$, $\lambda \leq  \xi_1 -1$ that
  \begin{equation*}
    \begin{split}
      \| (H_N - \lambda)^{-1/2} f\|_{H^1(\Omega)}^2 &\leq (\mathfrak{h}_N -\lambda)[(H_N - \lambda)^{-1/2} f] + \| (H_N - \lambda)^{-1/2} f \|_{L^2(\Omega)}^2 \\
      &\leq 2 \| f \|_{L^2(\Omega)}^2,
    \end{split}
  \end{equation*}
  which yields the claim in~\eqref{resolvent_power_bounded} for $a=1/2$. Thus, the statement for $a \in (0,1/2)$ follows from an interpolation argument, see, e.g., \cite{LM72}. Eventually, the claim in~\eqref{resolvent_power_bounded} is also true for $a>1/2$, as then $(H_N-\lambda)^{-a+1/2}$ is  uniformly bounded in $L^2(\Omega)$ in $\lambda \leq  \xi_1 -1$, and thus
  \begin{equation*}
    \begin{split}
    \|(H_N-\lambda)^{-a}&\|_{L^2(\Omega) \rightarrow H^1(\Omega)} \\
    &\leq \|(H_N-\lambda)^{-1/2}\|_{L^2(\Omega) \rightarrow H^1(\Omega)} \|(H_N-\lambda)^{-a+1/2}\|_{L^2(\Omega) \rightarrow L^2(\Omega)} \leq C.
    \end{split}
  \end{equation*}
  
  To proceed, denote by $\tau_D$ the Dirichlet trace defined on $H^{\min{\{1/2+2\delta,1\}}}$; cf.~\eqref{tauii}. 
  Clearly, $\tau_D$ is an extension of $\Gamma_1 = \widetilde{\Gamma}_1$, and taking~\eqref{resolvent_power_bounded} for $a = 1/4 + \delta$ into account, we find that
  \begin{equation*}
    C_2(\lambda) := \tau_D (H_N - \lambda)^{-1/4-\delta}: L^2(\Omega) \rightarrow L^2(\partial \Omega)
  \end{equation*}
  is well-defined and uniformly bounded with respect to $\lambda \in (-\infty,\xi_3)$, if $\xi_3$ is chosen smaller than $ \xi_1 -1$.
\end{proof}

In the next proposition we collect some properties of the Weyl functions corresponding to the generalized boundary triple 
$\{L^2(\partial\Omega), (\Gamma_0, \Gamma_1), (\widetilde\Gamma_0, \widetilde\Gamma_1)\}$. In particular, in item~(iv) we 
prove decay estimates for $M$ and $\widetilde{M}$.

\begin{proposition}\label{mprop}
Let Assumption~\ref{assi} be satisfied. 
Let $M$ and $\widetilde M$ be the Weyl functions corresponding to the generalized boundary triple 
$\{L^2(\partial\Omega), (\Gamma_0, \Gamma_1), (\widetilde\Gamma_0, \widetilde\Gamma_1)\}$ in Theorem~\ref{gbthurra}.
Then the following holds for all $\lambda\in\rho(A_0)$ and $\mu\in\rho(\widetilde A_0)$:
\begin{enumerate}[\upshape (i)]
 \item $M(\lambda)\tau_N f_\lambda=\tau_D f_\lambda$ for $f_\lambda\in H^{3/2}_\Delta(\Omega)$ such that $(-\Delta+V)f_\lambda=\lambda f_\lambda$;
 \item $\widetilde M(\mu)\tau_N g_\mu=\tau_D g_\mu$ for $g_\mu\in H^{3/2}_\Delta(\Omega)$ such that $(-\Delta+\myoverline V)g_\mu=\mu g_\mu$;
 \item $\ran M(\lambda)\subset H^1(\partial\Omega)$ and $\ran\widetilde M(\mu)\subset H^1(\partial\Omega)$, and, in particular, the operators 
 $M(\lambda)$ and $\widetilde M(\mu)$ are compact in $L^2(\partial\Omega)$;
 \item For all $\varepsilon > 0$ there exists a constant $C = C(\varepsilon)$ such that 
 $$\Vert M(\lambda)\Vert \leq C |\lambda|^{-1/2+\varepsilon}
 \quad\text{and}\quad \Vert \widetilde M(\mu)\Vert \leq C |\mu|^{-1/2+\varepsilon},
 \quad\lambda,\mu\rightarrow -\infty.$$
\end{enumerate}
\end{proposition}

\begin{proof}
Items (i) and (ii) are immediate consequences from Proposition~\ref{gam-m-prop}~(iv). It is also clear from the mapping properties of the Dirichlet trace 
$\tau_D$ in Theorem~\ref{traced} that 
$\ran M(\lambda)\subset H^1(\partial\Omega)$ and $\ran\widetilde M(\mu)\subset H^1(\partial\Omega)$. Furthermore, it is easy to check that $M(\lambda)$ 
and $\widetilde M(\mu)$ are closed as operators from $L^2(\partial\Omega)$ to $H^1(\partial\Omega)$ and hence bounded. Since $\partial\Omega$ is compact 
the embedding $H^1(\partial\Omega)\hookrightarrow L^2(\partial\Omega)$ is a compact operator and this implies (iii). 
In order to verify the claim on $M$ in (iv) (the assertion on $\widetilde{M}$ can be shown in a similar way) we employ the resolvent formula \eqref{tA0.res2} and Lemma~\ref{lemma_bounded}. In detail, with $\varepsilon>0$ small and $\lambda<0$ sufficiently negative, 
\begin{equation*} 
\begin{split}
M(\lambda)&=\Gamma_1\gamma(\lambda)^{**}\\
          &=\Gamma_1\bigl(\widetilde\Gamma_1(\widetilde A_0-\lambda)^{-1}\bigr)^*\\
          &=\Gamma_1\bigl(\widetilde\Gamma_1(H_N-\lambda)^{-1/2} (I+ C_1(\lambda)^*)^{-1}(H_N-\lambda)^{-1/2}\bigr)^*\\
          &=\Gamma_1\bigl(\widetilde\Gamma_1(H_N-\lambda)^{-1/4-\varepsilon/2}(H_N-\lambda)^{-1/4+\varepsilon/2} (I+ C_1(\lambda)^*)^{-1} 
          \\ & \hspace{4cm} \times (H_N-\lambda)^{-1/4+\varepsilon/2}(H_N-\lambda)^{-1/4-\varepsilon/2}\bigr)^*
          \\
          & = \Gamma_1 (H_N-\lambda)^{-1/4-\varepsilon/2} (H_N-\lambda)^{-1/4+\varepsilon/2} (I+ C_1(\lambda))^{-1}  
          \\ & \hspace{4cm} \times (H_N-\lambda)^{-1/4+\varepsilon/2} (\widetilde\Gamma_1(H_N-\lambda)^{-1/4-\varepsilon/2})^*.
\end{split}
\end{equation*}
By Lemma~\ref{lemma_bounded} the operator $\Gamma_1 (H_N-\lambda)^{-1/4-\varepsilon/2} = \widetilde\Gamma_1(H_N-\lambda)^{-1/4-\varepsilon/2}$ admits 
a uniformly bounded extension $C_2(\lambda)$, thus for all $\lambda<0$ sufficiently negative (recall that $\|C_1(\lambda)\|\leq 1/2$ from Proposition~\ref{prop:wt.A0.res})
\begin{equation*}
\|M(\lambda)\| \leq \|C_2(\lambda)\| |\lambda|^{-1/4+\varepsilon/2} 2 |\lambda|^{-1/4+\varepsilon/2} \|C_2(\lambda)^*\| = 2  \|C_2(\lambda)\|^2 |\lambda|^{-1/2+\varepsilon},
\end{equation*}
as claimed.
\end{proof}

Finally, we formulate a corollary of Theorem~\ref{bssatz} in the context of Schr\"odinger operators with complex potentials satisfying 
Assumption~\ref{assi}. For simplicity we assume that the parameter $B$ in the boundary condition is a bounded everywhere defined operator in $L^2(\partial\Omega)$, so that the condition $1\in\rho(BM(\lambda_0))$ in Theorem~\ref{bssatz} is satisfied for all 
$\lambda_0<0$ sufficiently negative
by
Proposition~\ref{mprop}~(iv).

\begin{corollary} \label{corollary_Robin} 
Let Assumption~\ref{assi} be satisfied and let
$B$ be a bounded everywhere defined operator in $L^2(\partial\Omega)$. 
Then 
\begin{equation*}
 A_B=-\Delta+ V,\quad \dom A_B=\bigl\{f\in H^{3/2}_\Delta(\Omega):\tau_N f=B\tau_D f\bigr\},
\end{equation*}
is a closed operator with a nonempty resolvent set, there exists $\xi_4<0$ such that $(-\infty,\xi_4) \subset \rho(A_0)\cap\rho(A_{B})$, and the Krein-type resolvent formula
\begin{equation}\label{Eq_Krein_formula}
(A_{B}-\lambda)^{-1}=(A_0-\lambda)^{-1}+\gamma(\lambda)\bigl(I-BM(\lambda)\bigr)^{-1}B\widetilde\gamma(\myoverline{\lambda})^*
\end{equation}
is valid for all $\lambda \in \rho(A_0) \cap \rho(A_B)$, where $\gamma$ and $\widetilde\gamma$ are the $\gamma$-fields associated to the generalized boundary triple 
$\{L^2(\partial\Omega), (\Gamma_0, \Gamma_1), (\widetilde\Gamma_0, \widetilde\Gamma_1)\}$ and $M$ is the Weyl function.
\end{corollary}

For completeness we remark that the operator $\widetilde\gamma(\myoverline{\lambda})^*=\Gamma_1(A_0- \lambda)^{-1}$ is closed and hence 
bounded as an operator from $L^2(\Omega)$ to $H^1(\partial\Omega)$. Since it is assumed that the boundary of the Lipschitz domain $\Omega$ is compact 
it follows that the embedding $H^1(\partial\Omega)\hookrightarrow L^2(\partial\Omega)$ is compact, and hence $\widetilde\gamma(\myoverline{\lambda})^*$
is a compact operator from $L^2(\Omega)$ to $L^2(\partial\Omega)$. Therefore, under the assumptions in Corollary~\ref{corollary_Robin} the perturbation term  
in the resolvent formula \eqref{Eq_Krein_formula} is compact in $L^2(\Omega)$ and it follows that $A_B$ is a compact perturbation of $A_0$ in resolvent 
sense.

\begin{appendix}

\section{Dirichlet and Neumann trace maps on Lipschitz domains with compact boundary}\label{appoja}

In this appendix we briefly collect some properties of the Dirichlet and Neumann trace maps on Lipschitz domains; for bounded Lipschitz domains 
the results are known from \cite{BGM25,GM11}. Recall first that for  
a bounded Lipschitz domain $\Omega \subset \R^n$, $n \ge 2$,
it is well known that for $s\in (1/2,3/2)$ the Dirichlet trace map $f\mapsto f|_{\partial\Omega}$ for $f\in C^\infty(\overline\Omega)$ admits a unique 
continuous extension
\begin{equation}\label{tauii}
		\tau_D: H^s(\Omega) \to H^{s-1/2}(\partial\Omega),\quad f\mapsto\tau_D f,\quad s\in (1/2,3/2),
	\end{equation}
and this extension has a continuous right inverse; cf. \cite[Theorem 3.38]{M00}. Similarly, if $\Omega$ is an unbounded Lipschitz domain with compact boundary
and $\chi:\Omega\rightarrow [0,1]$ is a smooth function equal to $1$ near $\partial\Omega$ and equal to $0$ sufficiently far away from $\partial\Omega$, then
one considers $\tau_D f:=\tau_D(\chi f)$, so that \eqref{tauii} extends naturally also to such domains.

Now it will be explained that the Dirichlet trace operator can be extended to the endpoints $s=1/2$ and $s=3/2$ if one assumes
some additional slight regularity in the $H^s$ spaces, that is, one considers the spaces $H_\Delta^{1/2}(\Omega)$ and $H_\Delta^{3/2}(\Omega)$ 
from \eqref{hs}. For bounded Lipschitz domains the next theorem is a variant of \cite[Theorem~3.6 and Corollary~3.7]{BGM25}, see also \cite[Lemma 3.1]{GM11}.
\begin{theorem}\label{traced}
Let Assumption~\ref{assi} \ref{asssi.Om} be satisfied.
%Let $\Omega \subset \R^n$, $n \ge 2$, be a \ps{Lipschitz domain with compact boundary}. 
Then for all $s\in [1/2, 3/2]$ the Dirichlet trace map \eqref{tauii} gives rise to a bounded, surjective operator
	\begin{equation}\label{eq:tauD.rest.HsDelta}
		\tau_D : H_\Delta^s(\Omega) \to H^{s-1/2}(\partial\Omega)
	\end{equation}
	(where $H_\Delta^s(\Omega)$ is equipped with the norm induced by \eqref{gn}), with bounded right-inverse. In addition, for each $s \in [1/2, 3/2]$, we have
	\begin{equation}\label{regd}
		\ker \tau_D \subset H^{3/2}(\Omega)
	\end{equation}
	and there exists $C > 0$ such that if $f \in H^{1/2}_{\Delta}(\Omega)$ and $\tau_D f = 0$, then $f \in H^{3/2}(\Omega)$ and
	\begin{equation}\label{jabadu}
		\| f \|_{H^{3/2}(\Omega)} \le C \bigl(\| f \|_{L^2(\Omega)} + \| \Delta f \|_{L^2(\Omega)}\bigr).
	\end{equation}
\end{theorem}

\begin{proof}
If $\Omega$ is a bounded Lipschitz domain, then the assertions are contained in \cite[Corollary 3.7]{BGM25}. In the case that $\Omega$ is unbounded with compact boundary the assertions follow with standard localization methods.
We briefly sketch the main arguments: First of all choose a $C^\infty(\Omega)$-function 
$\chi:\Omega\rightarrow [0,1]$ such that $\chi(x)=1$ for all $x$ near $\partial\Omega$ and $\chi(x)=0$ for all $x$ sufficiently far away from $\partial\Omega$.
For $f\in H_\Delta^s(\Omega)$ we obtain from $\Delta f\in L^2(\Omega)$ and standard elliptic regularity that $f\in H^2_{\textrm{loc}}(\Omega)$ and hence, as $\nabla \chi$ is compactly supported in $\Omega$,
$\nabla \chi \cdot \nabla f \in H^1(\Omega)$. 
Therefore,
\begin{equation} \label{distributional_Laplace}
\begin{split}
 \Delta (\chi f)&=(\Delta\chi) f + 2\nabla\chi\cdot\nabla f + \chi (\Delta f) \in L^2(\Omega),
\end{split}
 \end{equation}
and as $\chi f\in H^s(\Omega)$, see \cite[Theorem~3.20]{M00},
it follows that $\chi f\in H^s_\Delta(\Omega)$, and thus also $(1-\chi) f\in H^s_\Delta(\Omega)$. In particular, as $(1-\chi) f$ vanishes near $\partial\Omega$
one considers the Dirichlet trace 
$$\tau_D f:=\tau_D (\chi f),\qquad f=\chi f+ (1-\chi)f\in  H_\Delta^s(\Omega).$$ 
Then the assertions of the theorem follow when taking into account that $\chi f$ vanishes  
sufficiently far away from $\partial\Omega$, so that the properties for the Dirichlet trace on a bounded Lipschitz domain can be used. For 
the estimate \eqref{jabadu} one uses that 
\begin{equation} \label{chi_bounded}
  \| \chi g \|_{L^2(\Omega)} + \| \Delta (\chi g) \|_{L^2(\Omega)} \leq C \bigl( \| g \|_{L^2(\Omega)} + \| \Delta g \|_{L^2(\Omega)} \bigr), \quad g \in L^2_\Delta(\Omega) := H^0_\Delta(\Omega).
\end{equation}
Indeed, as in~\eqref{distributional_Laplace} one finds for $g \in L^2_\Delta(\Omega)$ that $\chi g \in L^2_\Delta(\Omega)$. Moreover, as $L^2_\Delta(\Omega)$ is continuously embedded in $L^2(\Omega)$ and the multiplication by $\chi$ gives rise to a bounded operator in $L^2(\Omega)$, one can verify that the multiplication by $\chi$ is a closed operator in $L^2_\Delta(\Omega)$. Therefore, by the closed graph theorem, one gets that~\eqref{chi_bounded} is true.
Therefore, using \eqref{jabadu} on bounded Lipschitz domains for $\chi f$ and~\eqref{chi_bounded} we conclude
\begin{equation} \label{jabadu1}
\begin{split}
 \Vert \chi f\Vert_{H^{3/2}(\Omega)}&\leq C \bigl(\| \chi f \|_{L^2(\Omega)} + \| \Delta (\chi f) \|_{L^2(\Omega)}\bigr)\\
 &\leq C \bigl(\| f \|_{L^2(\Omega)}+ \|  \Delta f \|_{L^2(\Omega)}\bigr).
 \end{split}
\end{equation}
To get a similar estimate for $(1-\chi) f$, note first that, as $1-\chi$ is zero in a neighborhood of $\partial \Omega$, one has for the zero extension $\widetilde{g}$ of $(1-\chi) f$ by standard elliptic regularity arguments that $\widetilde{g} \in H^2(\mathbb{R}^n)$. As the graph norm associated with $-\Delta$ in $\mathbb{R}^n$ is equivalent with the norm in $H^2(\mathbb{R}^n)$, the estimate 
\begin{equation*}
  \begin{split}
    \Vert (1-\chi) f\Vert_{H^{3/2}(\Omega)} &\leq \| \widetilde{g} \|_{H^2(\mathbb{R}^n)} \leq C \bigl( \| \Delta \widetilde{g} \|_{L^2(\mathbb{R}^n)} + \| \widetilde{g} \|_{L^2(\mathbb{R}^n)} \bigr) \\
    &= C \bigl( \| \Delta ((1-\chi) f) \|_{L^2(\Omega)} + \| (1-\chi) f \|_{L^2(\Omega)} \bigr)
  \end{split}
\end{equation*}
holds. Next, one finds as in~\eqref{chi_bounded} that the multiplication by $1-\chi$ gives rise to a bounded operator in $L^2_\Delta(\Omega)$. By combining this with the last displayed formula  one verifies $\Vert (1-\chi) f\Vert_{H^{3/2}(\Omega)}\leq C (\| f \|_{L^2(\Omega)}+ \|  \Delta f \|_{L^2(\Omega)})$, which
finally implies with~\eqref{jabadu1} the estimate in \eqref{jabadu} for unbounded domains.
\end{proof}

Now we turn to the Neumann trace operator $f\mapsto \nu\cdot \nabla f|_{\partial\Omega}$ for $f\in C^\infty(\Omega)\cap L^2(\Omega)$. 
For the case of a bounded Lipschitz domain $\Omega$ the next result is contained in \cite[Theorem 5.4 and Corollary 5.7]{BGM25} (see also \cite[Lemma 3.2]{GM11})
and for the case that $\Omega$ is unbounded with compact boundary the same localization arguments as in the proof of Theorem~\ref{traced}
can be applied to verify the statement; we leave the details to the reader. We recall that $(H^{t}(\partial\Omega))^*=H^{-t}(\partial\Omega)$ for $t\in [-1,1]$. 

\begin{theorem}\label{tracen}
Let Assumption~\ref{assi} \ref{asssi.Om} be satisfied.
%Let $\Omega \subset \R^n$, $n \ge 2$, be a \ps{Lipschitz domain with compact boundary}. 
Then for all $s\in [1/2, 3/2]$ the Neumann 
trace map induces a bounded, surjective operator
	\begin{equation}\label{eq:tauN.bgm.ext}
		\tau_N : H_\Delta^s(\Omega) \to H^{s-3/2}(\partial \Omega)
	\end{equation}
	(where $H_\Delta^s(\Omega)$ is equipped with the norm induced by \eqref{gn}), with bounded right-inverse. In addition, for each $s \in [1/2, 3/2]$, we have
		\begin{equation}\label{regn}
			\ker \tau_N \subset H^{3/2}(\Omega)
		\end{equation}
		and there exists $C > 0$ such that if $f \in H^{1/2}_{\Delta}(\Omega)$ and $\tau_N f = 0$, then $f \in H^{3/2}(\Omega)$ and
		\begin{equation}\label{boundn}
			\| f \|_{H^{3/2}(\Omega)} \le C \bigl(\| f \|_{L^2(\Omega)} + \| \Delta f \|_{L^2(\Omega)}\bigr).
		\end{equation}
		Furthermore, if $f \in H_\Delta^{3/2}(\Omega)$, then $\tau_N f = \nu \cdot \tau_D (\nabla f)$.%
\end{theorem}

Next, we recall a version of the second Green identity for the trace operators in Theorem~\ref{traced} and Theorem~\ref{tracen}.
Observe first that for $s \in [1/2, 3/2]$ and $f\in H_\Delta^s(\Omega)$ we have 
\begin{equation}\label{ddd}
 \tau_D f\in  H^{s-1/2}(\partial\Omega)\quad\text{and}\quad \tau_N f\in  H^{s-3/2}(\partial\Omega)=\bigl(H^{3/2-s}(\partial\Omega)\bigr)^*,
\end{equation}
and in the same way for $g \in H_\Delta^{2-s}(\Omega)$ we have 
\begin{equation}\label{nnn}
 \tau_D g\in  H^{3/2-s}(\partial\Omega)\quad\text{and}\quad \tau_N g\in  H^{1/2-s}(\partial\Omega)=\bigl(H^{s-1/2}(\partial\Omega)\bigr)^*.
\end{equation}
Then for $s \in [1/2, 3/2]$ and $f\in H_\Delta^s(\Omega)$, $g \in H_\Delta^{2-s}(\Omega)$ one has 
		\begin{equation}\label{ggg22}
			\begin{aligned}
				(-\Delta f, g)_{L^2(\Omega)} - (f, -\Delta g)_{L^2(\Omega)} &= \langle \tau_D f, \tau_N g \rangle_{H^{s-1/2}(\partial\Omega) \times (H^{s-1/2}(\partial\Omega))^*}\\
				&\quad - \langle \tau_N f, \tau_D g \rangle_{(H^{3/2-s}(\partial\Omega))^* \times H^{3/2-s}(\partial\Omega)};
			\end{aligned}
		\end{equation}
		cf. \cite[Corollary 5.7]{BGM25} for bounded $\Omega$; the case of unbounded Lipschitz domains with compact boundary can again be handled with a localization argument as in the proof of Theorem~\ref{traced}. In particular, for $f,g\in H_\Delta^{3/2}(\Omega)$ the traces in \eqref{ddd} and \eqref{nnn} 
		are contained in $L^2(\partial\Omega)$ and \eqref{ggg22} takes the form
		\begin{equation}\label{greeni}
				(-\Delta f, g)_{L^2(\Omega)} - (f, -\Delta g)_{L^2(\Omega)} = ( \tau_D f, \tau_N g) _{L^2(\partial\Omega)} - 
				( \tau_N f, \tau_D g)_{L^2(\partial\Omega)}.
		\end{equation}
		
Furthermore, for $f\in H_\Delta^{1}(\Omega)$ and $g\in H^1(\Omega)$ the first Green identity
\begin{equation*}
(-\Delta f, g)_{L^2(\Omega)}=(\nabla f,\nabla g)_{L^2(\Omega;\R^n)}-\langle \tau_N f, \tau_D g \rangle_{H^{-1/2}(\partial\Omega) \times H^{1/2}(\partial\Omega)}
\end{equation*}
holds; cf. \cite[Theorem 4.4~(i)]{M00}. If, in addition, $f\in H_\Delta^{3/2}(\Omega)$, then $\tau_N f$ is contained in $L^2(\partial\Omega)$ and one has 
\begin{equation}\label{usethissomewhere}
(-\Delta f, g)_{L^2(\Omega)}=(\nabla f,\nabla g)_{L^2(\Omega;\R^n)}- ( \tau_N f, \tau_D g )_{L^2(\partial\Omega)}.
\end{equation}

In the next lemma we collect an additional regularity property for the functions $f\in  H_\Delta^s(\Omega)$ that satisfy $\tau_D f=\tau_N f=0$.
The assertion follows from \eqref{regd}, \eqref{regn}, and \cite[Theorem 6.12 and Remark 5.8]{BGM25} for bounded Lipschitz domains and extends with the help of localization arguments as in the proof of Theorem~\ref{traced} also to unbounded Lipschitz domains with compact boundary.
		
\begin{lemma}\label{reglem}
Let Assumption~\ref{assi} \ref{asssi.Om} be satisfied
%Let $\Omega \subset \R^n$, $n \ge 2$, be a \ps{Lipschitz domain with compact boundary} 
and let $s\in [1/2, 3/2]$. Then the Dirichlet and Neumann trace operators 
	\begin{equation*}
		\tau_D : H_\Delta^s(\Omega) \to H^{s-1/2}(\partial \Omega)\quad\text{and}\quad \tau_N : H_\Delta^s(\Omega) \to H^{s-3/2}(\partial \Omega)
	\end{equation*}
	satisfy $\ker\tau_D\cap\ker\tau_N=H^2_0(\Omega)$.
\end{lemma}
		
\end{appendix}		

%%%%%%%%%%%%%%%%%%%%%%%%%%%%%%
\subsection*{Funding}  
This research was funded in part by the Austrian Science Fund (FWF) 10.55776/P 33568-N. For the purpose of open access, the author has applied a CC BY public copyright licence to any Author Accepted Manuscript version arising from this submission.

\subsection*{Data availability statement}

Data sharing not applicable to this article as no datasets were generated or analysed during the current study.

\subsection*{Competing Interests}

The authors have no competing interests to declare that are relevant to the content of this article.

\end{document}